\documentclass[letterpaper, 10 pt, conference]{ieeeconf}  

\IEEEoverridecommandlockouts                              

\usepackage[letterpaper,top=54pt,bottom=58pt,left=54pt,right=54pt]{geometry}

\usepackage{graphics} 
\usepackage{epsfig} 
\usepackage{mathptmx} 
\usepackage{times} 
\usepackage{amsmath} 
\usepackage{amssymb}  
\usepackage{xcolor}

\newtheorem{remark}{Remark}
\newtheorem{theorem}{Theorem}

\newtheorem{prop}{Proposition}
\begin{document}

\title{\LARGE \bf 
Sub-predictors \textcolor{blue}{and} classical predictors for finite-dimensional observer-based control of parabolic PDEs}
\author{Rami Katz and Emilia Fridman,~\IEEEmembership{Fellow,~IEEE}
\thanks{Supported by  Israel Science Foundation (grant no. 673/19) and
	by Chana and Heinrich Manderman Chair at Tel Aviv University.}
\thanks{R. Katz ({\tt\small rami@benis.co.il}) and E. Fridman ({\tt\small  emilia@eng.tau.ac.il}) are with the School of Electrical Engineering, Tel Aviv University, Israel.}
}
\maketitle

\begin{abstract}
We study constant input delay compensation by using  finite-dimensional observer-based controllers in the case of the 1D heat equation. We consider Neumann actuation with nonlocal measurement and employ modal decomposition with $N+1$ modes in the observer. We introduce a chain of $M$ sub-predictors that leads to a closed-loop ODE system
coupled with infinite-dimensional tail. Given an input delay $r$, we present LMI stability conditions for finding $M$ and $N$ and the resulting exponential decay rate and prove that the LMIs are always feasible for
any $r$. We also consider a classical observer-based predictor and show that the corresponding LMI stability conditions are feasible for any $r$ provided $N$ is large enough.
A numerical example demonstrates that the classical predictor leads to a lower-dimensional observer. However, it is known to be hard for implementation due to the distributed input signal.
\end{abstract}

\section{Introduction}
Finite-dimensional observer-based  controllers for parabolic systems were designed by the modal decomposition approach in \cite{balas1988finite, christofides2001, curtain1982finite,harkort2011finite,
nambu2017theory}.
Recently, the first constructive LMI-based method for finite-dimensional observer-based controller was suggested in \cite{RamiContructiveFiniteDim}
for the 1D heat equation under nonlocal or Dirichlet actuation and nonlocal measurement.
The observer dimension $N$ and the resulting exponential decay rate
were found from simple LMI conditions.
Finite-dimensional observer-based control of the  Kuramoto-Sivashinsky equation with boundary actuation and point
measurement was studied in \cite{Rami_CDC20}.

Robustness with respect to small delays and/or sampling intervals for the heat equation was studied in \cite{Aut12,NetzerAut14} 
for distributed static output-feedback control, in \cite{karafyllis2018sampled} for boundary state-feedback and in \cite{selivanov2018delayed,katz2020boundary}
for boundary  controller based on PDE observer.
Delayed implementation of finite-dimensional observer-based controllers for the 1D heat equation
was introduced in \cite{katz2020constructiveDelay}, where in case of time-varying output delay, a combination of Lyapunov functionals with Halanay's inequality appeared to be an efficient tool.

To compensate large input/output delay,
there are two main predictor  methods: the classical predictor, which is based on a reduction approach \cite{Artstein82} or  the backstepping approach \cite{Krstic09} and sub-predictors or chain of observers \cite{germani2002new,ahmed2011cascade,najafi2013closed,zhu2021sub}. The classical predictors for state-feedback control of PDEs were suggested  in \cite{Krstic09,sano2018neumann,lhachemi2019lmi}. For the heat equation, a PDE sub-predictor (an observer of the future state) was presented in \cite{selivanov2018delayed}. A chain of observers for the estimation of heat equation with a large output delay was designed in \cite{ahmed2019observer}.

In the recent paper \cite{katzreduced21}, reduced-order LMI stability conditions were introduced for finite-dimensional observer-based control. This  was presented for the heat equation with Neumann
actuation and non-local measurement. The dimension of the LMIs does not grow with the dimension of the observer $N$. Moreover, feasibility of the LMIs for $N$ implies their feasibility for $N+1$.
In \cite{katzreduced21}, the classical predictor was extended to finite-dimensional observer-based control. This predictor compensated delay in the finite-dimensional controller, whereas the  infinite-dimensional part still depended on the large input delay. It was shown in a numerical example that the predictor allows for larger delays. However, the feasibility of LMIs for arbitrary delays was not proved due to complexity of the analysis in the presence of time-varying output delay.

The present paper is dedicated to predictor methods for finite-dimensional observer-based control of parabolic PDEs with constant input delay $r$.
As in \cite{katzreduced21}, we consider the 1D heat equation under Neumann actuation and non-local measurement. The main novelty is in use of sub-predictors for such a system. We show that for any $r$ there exists a chain of $M$ sub-predictors and a large enough number of modes $N+1$  employed in observer that guarantee the stability of the  closed-loop system. We present LMI stability conditions for finding $M$, $N$ and the resulting exponential decay rate.
We prove that these LMIs are always feasible for all $r$ and large enough $M$ and $N$. We also consider the classical predictor which compensates the delay in the finite-dimensional part, as introduced in \cite{katzreduced21} (if the time-varying input/ouput delays are omitted). \textcolor{blue}{This is the first time that feasibility guarantees for the resulting LMIs with arbitrary delays are proved for both sub-predictors and predictors. This proof is challenging, due to coupling in the closed-loop system.} A numerical example demonstrates that for the same $N$, the classical predictor allows larger delays found from the LMIs, whereas for the same delay they employ lower-dimensional observers than the sub-predictors. However, as is well-known \cite{MondieWim03,karafyllis2017predictor,furtat2017disturbance}, they are harder to implement, due to the distributed input term 
which should be carefully discretized. \textcolor{blue}{This paper is an essential step towards the use of sub-predictors and classical predictors for delay compensation in PDEs, via finite-dimensional observers}.

\vspace{0.1cm}
\emph{Notations and preliminaries:}
$ L^2(0,1)$ is the Hilbert space of Lebesgue measurable and square integrable functions $f:[0,1]\to \mathbb{R} $  with the inner product $\left< f,g\right>:=\scriptsize{\int_0^1 f(x)g(x)dx}$ and induced norm $\left\|f \right\|^2:=\left<f,f \right>$.
$H^1(0,1)$ is the space of functions  $f:[0,1]\to \mathbb{R} $ with square integrable weak derivative, with the norm $\left\|f \right\|^2_{H^1}:=\sum_{j=0}^{1} \left\|f^{(j)} \right\|^2$.
The Euclidean norm on $\mathbb{R}^n$ is denoted by $\left|\cdot \right|$. For  $P \in \mathbb{R}^{n \times n}$, $P>0$ means that $P$ is symmetric and positive definite. The sub-diagonal elements of a symmetric matrix will be denoted by $*$.  $\otimes$ is the standard Kronecker product. For $U\in \mathbb{R}^{n\times n}, \ U>0$ and $x\in \mathbb{R}^n$ let $\left|x\right|^2_U=x^TUx$. $\mathbb{Z}_+$ is the set of nonnegative integers.

Recall that the Sturm-Liouville eigenvalue problem
\begin{equation}\label{eq:SL}
	\begin{array}{lll}
		&\phi''+\lambda \phi = 0,\ \ x\in [0,1],\ \ \phi'(0)=\phi'(1)=0,
	\end{array}
\end{equation}
induces a sequence of eigenvalues $\scriptsize \lambda_n = n^2\pi^2, n\in \mathbb{Z}_+$ with corresponding eigenfunctions
\begin{equation}\label{eq:Eigenfunctions}
	\begin{array}{lll}
		&\phi_0(x) = 1, \ \phi_n(x)=\sqrt{2}\cos\left(\sqrt{\lambda_n} x\right), n\geq 1.
	\end{array}
\end{equation}
The eigenfunctions form a complete orthonormal system in $L^2(0,1)$.
Given $N\in \mathbb{Z}_+$ and $h\in L^2(0,1)$ satisfying $h \overset{L^2}{=} \sum_{n=0}^{\infty}h_n\phi_n$ we denote
$\left\|h\right\|_N^2 =\sum_{n=N+1}^{\infty}h_n^2.$
\section{Sub-predictors vs classical predictors}
We consider the PDE
\begin{equation}\label{eq:PDEPointActIntervalDelayedSub}
	\begin{aligned}
		& z_t(x,t)=z_{xx}(x,t)+qz(x,t),\ x\in[0,1],\ t\geq 0,\\
		& z_x(0,t)=0, \quad z_x(1,t)=u(t-r)
	\end{aligned}
\end{equation}
under delayed Neumann actuation with known delay $r$ and non-local measurement
\begin{equation}\label{eq:BoundMeasPred}
	y(t) = \left<c,z(\cdot,t)\right>, \ t\geq 0
\end{equation}
with $c\in L^2(0,1)$. To compensate the delay, we will present in this section both sub-predictors and classical predictors.

Using modal decomposition, we present the solution to \eqref{eq:PDEPointActIntervalDelayedSub} as
\begin{equation}\label{eq:zSolDecompSub}
	z(x,t)\overset{L^2}{=}\sum_{n=0}^{\infty}z_n(t)\phi_n(x), \ \  z_n(t) = \left<z(\cdot,t),\phi_n\right>
\end{equation}
with $\phi_n, \ n\in \mathbb{Z}_+$ given in \eqref{eq:Eigenfunctions}. Differentiating under the integral, integrating by parts and using \eqref{eq:SL} and \eqref{eq:Eigenfunctions} we obtain {\color{blue}(similar to \cite{karafyllis2018sampled} and the references therein)}
\begin{equation}\label{eq:znDelayedSub}
	\begin{aligned}
		&\dot{z}_n(t) = (- \lambda_n + q) z_n(t) + b_n u(t-r),\ t\geq 0 \\
		&b_0 = 1, \ \ b_n = {(-1)}^n \sqrt{2}, \ n\geq 1.
	\end{aligned}
\end{equation}
Let $\delta>0$ be a desired decay rate. Since $\lim_{n \to \infty}\lambda_n=\infty$, there exists some $N_0 \in \mathbb{Z}_+$ such that
\begin{equation}\label{eq:N0Sub}
	-\lambda_n+q<-\delta, \quad n>N_0.
\end{equation}	

Let
\begin{equation}\label{eq:C0A0Sub}
	\begin{aligned}
		&A_0 = \operatorname{diag}\left\{-\lambda_0+q,\dots,-\lambda_{N_0}+q \right\},\\
		&L_0 = \left[l_0,\dots,l_{N_0} \right]^T,\ B_0:=\begin{bmatrix}b_0& \dots & b_{N_0} \end{bmatrix}^T\\
		&C_0=\left[c_0,\dots,c_{N_0} \right],\ c_n=\left<c,\phi_n\right>, \ \ n \in \mathbb{Z}_+.
	\end{aligned}
\end{equation}
Assume that
\begin{equation}\label{eq:AsscnNonDelayedSub}
	c_n \neq 0 , \quad 0\leq n \leq N_0.
\end{equation}
Then $(A_0,C_0)$ is observable, by the Hautus lemma. We choose $L_0 = [l_0,\dots, l_{N_0}]^T$ which satisfies the following Lyapunov inequality:
\begin{equation}\label{eq:GainsDesignLSub}
	P_{\text{o}}(A_0-L_0C_0)+(A_0-L_0C_0)^TP_{\text{o}} < -2\delta P_{\text{o}},
\end{equation}
where $0<P_{\text{o}}\in \mathbb{R}^{(N_0+1)\times (N_0+1)}$.

Similarly, by the Hautus lemma, $b_n\neq 0,\ n\in \mathbb{Z}_+$ implies that $(A_0,B_0)$ is controllable. Let $K_0\in \mathbb{R}^{1\times (N_0+1)}$ satisfy
\begin{equation}\label{eq:GainsDesignKSub}
	\begin{aligned}
		&P_{\text{c}}(A_0-B_0K_0)+(A_0-B_0K_0)^TP_{\text{c}} < -2\delta P_{\text{c}},
	\end{aligned}
\end{equation}
where $0<P_{\text{c}}\in \mathbb{R}^{(N_0+1)\times (N_0+1)}$.

 {\color{blue} In our finite-dimensional observer-based predictor design, the closed-loop system will be presented as a coupled system of ODEs and the infinite-dimensional tail. This complicates the proof of stabilization for all $r>0$ under higher-dimensional observers.}

Given $N\geq N_0$ denote
\begin{equation}\label{eq:ObsLastModesSub}
	\begin{array}{lll}
		&{z}^{N_0}(t) = \text{col}\left\{z_i\right\}_{i=0}^{N_0}, \ \hat{z}^{N-N_0}(t)= \text{col}\left\{\hat{z}_i\right\}_{i=N_0+1}^N,\\
		& A_1 =  \operatorname{diag}\left\{-\lambda_{N_0+1}+q,\dots,-\lambda_N+q \right\},\\
		& B_1 = \left[b_{N_0+1},\dots,b_N \right]^T,\ C_1=\left[c_{N_0+1},\dots,c_{N} \right].
	\end{array}
\end{equation}
\subsection{Sub-predictors}
In order to deal with a large delay $r$, we subdivide $r$ into $M$ parts of equal size $\frac{r}{M}$, where $M\in \mathbb{Z}_+, \ M\geq 1$.
We first consider $M\ge 2$ and  employ a chain of sub-predictors (observers of the future state)
\begin{equation}\label{eq:ChainSubPred}
	\begin{array}{lll}
		&\hat{z}^{N_0}_1(t-r)\mapsto \cdots\mapsto\hat{z}^{N_0}_i\left(t-\frac{M-i+1}{M}r\right)\mapsto\cdots\\
		&\mapsto\hat{z}^{N_0}_{M}\left(t-\frac{1}{M}r\right)\mapsto {z}^{N_0}(t).
	\end{array}
\end{equation}
Here  $\hat{z}^{N_0}_i\left(t-\frac{M-i+1}{M}r\right)\mapsto\hat{z}^{N_0}_{i+1}\left(t-\frac{M-i}{M}r\right)$ means that $\hat{z}^{N_0}_i(t)$ predicts  the value of $\hat{z}^{N_0}_{i+1}(t+\frac{r}{M})$. Similarly, $\hat{z}^{N_0}_M(t)$ predicts the value of $z^{N_0}(t+\frac{r}{M})$. The sub-predictors satisfy the following ODEs for $t\geq 0$
\begin{equation}\label{eq:SubpredODEs}
	\begin{array}{lll}
		&\dot{\hat{z}}^{N_0}_M(t)=A_0\hat{z}^{N_0}_M(t)+B_0u\left(t-\frac{M-1}{M}r\right)\\
		&\hspace{3mm}-L_0\left[C_0\hat{z}^{N_0}_M\left(t-\frac{r}{M}\right)+C_1\hat{z}^{N-N_0}(t)-y(t) \right],\\
		&\dot{\hat{z}}^{N_0}_i(t) = A_0\hat{z}^{N_0}_i(t)+B_0u\left(t-\frac{i-1}{M}r \right)\\
		&\hspace{3mm}-L_0C_0\left[\hat{z}_i^{N_0}\left(t-\frac{r}{M}\right)-\hat{z}^{N_0}_{i+1}(t) \right], \ 1\leq i\leq M-1,\\
		& \hat{z}^{N_0}_i(t)=0, \quad t\leq 0, \ 1\leq i\leq M,
	\end{array}
\end{equation}
whereas
$\hat{z}^{N-N_0}(t)$ satisfies the following ODE
\begin{equation}\label{eq:obsODENonDelayedSubLastODEs}
	\begin{array}{lll}
		&\dot{\hat{z}}^{N-N_0}(t) = A_1\hat{z}^{N-N_0}(t) + B_1u(t-r),\\
		&\hat{z}^{N-N_0}(t)=0, \quad t\leq 0.
	\end{array}
\end{equation}
The finite-dimensional observer $\hat z(x,t)$ of the state $z(x,t)$, based on $(M-1)(N_0+1)+N+1$-dimensional system of ODEs \eqref{eq:SubpredODEs}-\eqref{eq:obsODENonDelayedSubLastODEs}, is given by
\begin{equation}\label{eq:ZhatSeries0Sub}
	\begin{array}{lll}
		&\hat{z}(x,t)= \hat{z}^{N_0}_1(t-r)\cdot \text{col}\left\{\phi_0(x),\dots,\phi_{N_0}(x)\right\}\\
		&\hspace{8mm}+\hat{z}^{N-N_0}(t)\cdot\text{col}\left\{\phi_{N_0+1}(x),\dots,\phi_{N}(x)\right\}.
	\end{array}
\end{equation}
The controller is further chosen as
\begin{equation}\label{eq:ControllerSub}
	u(t)=-K_0{\hat{z}}^{N_0}_1(t).
\end{equation}
In particular, \eqref{eq:SubpredODEs} implies $u(t)=0$ for $t\leq 0$.

For well-posedness we introduce the change of variables $w(x,t)=z(x,t)-\frac{1}{2}x^2u(t-r)$. Then, the closed-loop system is presented as
\begin{equation}\label{eq:PDEChangeVar}
	\begin{array}{lll}
		&\hspace{-5mm}w_t(x,t)=w_{xx}(x,t)+qw(x,t)+f(x,t), \ x\in [0,1], \ t\geq 0,\\
		&\hspace{-5mm}w_x(0,t)=0, \quad w_x(1,t)=0,\\
		&\hspace{-5mm}f(x,t) = -\frac{1}{2}x^2\dot{u}(t-r)+\left(\frac{q}{2}x^2+1\right)u(t-r),
	\end{array}
\end{equation}
the ODEs \eqref{eq:SubpredODEs} and \eqref{eq:ControllerSub}. Let $z(\cdot,0)=w(\cdot,0)\in H^1(0,1)$. We apply the step method on $\left\{[jr,(j+1)r]\right\}_{j=0}^{\infty}$. For $t\in [0,r]$ we have that $f(x,t)\equiv 0$. By Theorems 6.3.1 and 6.3.3 in \cite{pazy1983semigroups}, \eqref{eq:PDEChangeVar} has a unique classical solution $z=w\in C([0,r],L^2(0,1))\cap C^1((0,r],L^2(0,1))$ such that $w(\cdot,t)\in H^2(0,1)$ with $w_x(0,t)=w_x(1,t)=0$ for $t\in(0,r]$. Furthermore, since $u(t-r)\equiv 0$ for $t\in [0,r]$, \eqref{eq:obsODENonDelayedSubLastODEs} implies that $\hat{z}^{N-N_0}(t)\in C^1[0,r]$. Since $y\in C[0,r]$, considering  \eqref{eq:SubpredODEs} on the
subintervals $\left\{[\frac{j}{M}r,\frac{(j+1)}{M}r]\right\}_{j=0}^{M-1}$, it can be seen that $\hat{z}^{N_0}_i\in C^1[0,r], \ 1\leq i \leq M$.
Furthermore, $\dot{\hat{z}}^{N_0}_1$ is Lipschitz for $t\in [0,r]$. Next, we consider $t\in [r,2r]$. Since $\hat{z}^{N_0}_1(t)\in C^1[0,r]$, with $\dot{\hat{z}}^{N_0}_1(t)$  Lipschitz on $[0,r]$, we have that $f(x,t)$ is Lipschitz on $[r,2r]$. By Theorems 6.3.1 and 6.3.3 in \cite{pazy1983semigroups}, \eqref{eq:PDEChangeVar} has a unique classical solution for $t\in[r,2r]$. Continuing step-by-step and using $z(x,t)=w(x,t)+\frac{1}{2}x^2u(t-r)$, \eqref{eq:PDEPointActIntervalDelayedSub} has a unique solution
$z\in C([0,\infty),L^2(0,1))\cap C^1((0,\infty)\setminus \mathcal{J},L^2(0,1))$, where $\mathcal{J}=\left\{\frac{jr}{M}\right\}_{j=0}^{\infty}$. Moreover, $z(\cdot,t)\in H^2(0,1)$ with $z_x(0,t)=0, \ z_x(1,t)=u(t-r)$ for $t\in[0,\infty)$.

Define the estimation errors for $1\leq i \leq M-1$ as follows:
\begin{equation}\label{eq:EstErrSub}
	\begin{array}{lll}
		&e^{N-N_0}(t) = \text{col}\left\{z_{N_0+1}(t),\dots,z_N(t)\right\}-\hat{z}^{N-N_0}(t),\\
		&e^{N_0}_M(t)={z}^{N_0}(t)-\hat{z}^{N_0}_{M}\left(t-\frac{1}{M}r\right),\\
		& e^{N_0}_i(t) = \hat{z}^{N_0}_{i+1}\left(t-\frac{M-i}{M}r\right)-\hat{z}^{N_0}_{i}\left(t-\frac{M-i+1}{M}r\right).
	\end{array}
\end{equation}
From \eqref{eq:SubpredODEs} and \eqref{eq:EstErrSub} we have
\begin{equation}\label{eq:NoDelayControlPresent}
	\begin{array}{l}
		\hat z_1^{N_0}(t-r)+\sum_{i=1}^Me^{N_0}_i(t)=z^{N_0}(t).
	\end{array}
\end{equation}
In particualr, if the errors $e^{N_0}_i(t), \ 1\leq i \leq M$ converge to zero, we have $\hat{z}^{N_0}_1(t)\rightarrow {z}^{N_0}(t+r)$, meaning that $\hat{z}_1^{N_0}(t)$ sequentially forecasts the future system state $z^{N_0}(t+r)$. Using \eqref{eq:BoundMeasPred}, \eqref{eq:C0A0Sub} and \eqref{eq:EstErrSub}, the innovation term in the ODE for $\hat{z}^{N_0}_M(t)$ (see \eqref{eq:SubpredODEs}), can be presented as
\begin{equation}\label{eq:InnovTermSub}
	\begin{array}{lll}
		&C_0\hat{z}^{N_0}_M\left(t-\frac{r}{M}\right)+C_1\hat{z}^{N-N_0}(t)-y(t)\\
		&\hspace{10mm} = -C_0e^{N_0}_M(t)-C_1e^{N-N_0}(t)-\zeta(t),\\
		&\zeta(t) = \sum_{n=N+1}^{\infty}c_nz_n(t).
	\end{array}
\end{equation}
By the Cauchy-Schwarz inequality we have
\begin{equation}\label{eq:ZetaEstBoundartActSub}
	\begin{array}{lll}
		\zeta^2(t) &\leq \left\|c \right\|_{N}^2\sum_{n=N+1}^{\infty}z_n^2(t).
	\end{array}
\end{equation}
Using \eqref{eq:znDelayedSub}, \eqref{eq:SubpredODEs} and \eqref{eq:InnovTermSub} we obtain the following dynamics of the estimation errors for $t\geq 0$
\begin{equation}\label{eq:ErrEqSub}
	\begin{array}{lll}
		&\dot{e}_M^{N_0}(t) = A_0e^{N_0}_M(t)-L_0C_0e^{N_0}_M(t-\frac{r}{M})\\
		&\hspace{10mm}-L_0C_1e^{N-N_0}\left(t-\frac{r}{M}\right)-L_0\zeta(t-\frac{r}{M}),\\
		&\dot{e}_{M-1}^{N_0}(t) = A_0e^{N_0}_{M-1}(t)-L_0C_0e^{N_0}_{M-1}\left(t-\frac{r}{M}\right)\\
		&\hspace{10mm}+L_0C_0e^{N_0}_M\left(t-\frac{r}{M}\right)+L_0C_1e^{N-N_0}\left(t-\frac{r}{M} \right)\\
		&\hspace{10mm}+L_0\zeta(t-\frac{r}{M}),\\
		&\dot{e}_{i}^{N_0}(t) = A_0e^{N_0}_i(t)-L_0C_0e^{N_0}_i\left(t-\frac{r}{M}\right)\\
		&\hspace{10mm}+L_0C_0e^{N_0}_{i+1}\left(t-\frac{r}{M}\right), \ 1\leq i \leq M-2
	\end{array}
\end{equation}
and
\begin{equation}\label{eq:ErrEqSub1}
	\dot{e}^{N-N_0}(t) = A_1e^{N-N_0}(t).
\end{equation}
From \eqref{eq:znDelayedSub}, \eqref{eq:ControllerSub} and \eqref{eq:NoDelayControlPresent},  $z^{N_0}(t)$ satisfies
\begin{equation}\label{eq:dzn0clM}
	\begin{aligned}
		\dot {z}^{N_0}(t) = (A_0-B_0K_0) z^{N_0}(t)+ B_0K_0\sum_{i=1}^Me^{N_0}_i(t).
	\end{aligned}
\end{equation}
We introduce the notations
\begin{equation*}\label{eq:notationsDelayMSubNew}
	\begin{array}{lll}
		&\hspace{-3mm} X_e(t) = \text{col}\left\{e^{N_0}_i(t)\right\}_{i=1}^M,\ \nu_e(t) = X_e\left(t-\frac{r}{M}\right)-X_e(t),\\
		&\hspace{-3mm}F_e = \operatorname{diag}\scriptsize\left\{I_{M-1}\otimes (A_0-L_0C_0)+ J_{M-1}(0)\otimes L_0C_0, A_0-L_0C_0\right\}, \\
		&\hspace{-3mm}G_e = \operatorname{diag}\scriptsize\left\{I_{M-1}\otimes (-L_0C_0)+ J_{M-1}(0)\otimes L_0C_0,-L_0C_0\right\}, \\
		&\hspace{-3mm}\mathcal{L}_e = \text{col}\left\{0_{(M-2)(N_0+1) \times 1}, L_0,-L_0\right\},\\
		&\hspace{-3mm}\mathcal{K}_e = \begin{bmatrix} K_0, ...,  K_0\end{bmatrix}\in \mathbb{R}^{1\times M(N_0+1)}
	\end{array}
\end{equation*}
Where $J_{M-1}(0)$ is a Jordan block of order $M-1$ with zero diagonal. Note that \eqref{eq:ErrEqSub1} implies
$e^{N-N_0}\left(t-{\scriptsize\frac{r}{M}}\right) = e^{-A_1\frac{r}{M}}e^{N-N_0}(t).$
Then, using \eqref{eq:znDelayedSub}, \eqref{eq:ControllerSub}, \eqref{eq:ErrEqSub} and \eqref{eq:dzn0clM}, the reduced-order (i.e, decoupled from $\hat{z}^{N-N_0}(t)$) closed-loop system can be presented as
\begin{equation}\label{eq:XMDelayNew}
	\begin{array}{llllll}
		& \dot {z}^{N_0}(t) = (A_0-B_0K_0) z^{N_0}(t)+ B_0\mathcal{K}_eX_e(t) \\
		&\dot{X}_e(t) =  F_eX_e(t)+G_e \nu_e(t) +\mathcal{L}_e\zeta\left(t-\frac{r}{M}\right)\\
		&\hspace{10mm} +\mathcal{L}_e C_1e^{-A_1\frac{r}{M}}e^{N-N_0}(t),\\
		&\dot{z}_n(t)= (-\lambda_n+q)z_n(t)-b_nK_0z^{N_0}(t),\\
		&\hspace{10mm}+b_n\mathcal{K}_eX_e(t),\quad  n>N.
	\end{array}
\end{equation}
In the case $M=1$, $\hat{z}^{N_0}(t)$ satisfies the first ODE in \eqref{eq:SubpredODEs} and predicts $z^{N_0}(t+r)$. Here $X_e(t)= e^{N_0}_1(t)$ and the closed-loop system has the form \eqref{eq:ErrEqSub1} and \eqref{eq:XMDelayNew}, where now
\begin{equation*}
	F_e = A_0-L_0C_0,\ G_e = -L_0C_0, \ \mathcal{K}_e=K_0, \ \mathcal{L}_e = -L_0.
\end{equation*}

Differently from the existing finite-dimensional controllers \cite{RamiContructiveFiniteDim,katz2020constructiveDelay}, where the closed-loop systems is written in terms of the observer and the tail $z_n(t) \ (n>N)$,
here 
\eqref{eq:XMDelayNew} is presented in terms of the state $z^{N_0(t)}$, the estimation errors $X_e(t)$ and the tail. This allows to eliminate the delay $r$ from the ODEs of $z^{N_0}(t)$ and $z_n(t), \ n>N$ while decreasing it to  $\frac{r}{M}$ (which is small for large $M$) in the ODE of $X_e(t)$. 
{\color{blue}
\begin{remark}\label{rem_ODEs}
In the case of sub-predictors for linear ODEs, the closed-loop system is given by \eqref{eq:ErrEqSub} and \eqref{eq:dzn0clM}, where $\zeta=0$ and $e^{N-N_0}=0$.
Thus, exponential stability of
\begin{equation}\label{eq:e_M}
	\dot{e}_M^{N_0}(t) = (A_0-L_0C_0)e_M^{N_0}(t)-L_0C_0\nu_{e,M}(t),
\end{equation}
where $\nu_{e,M}(t)=e_M^{N_0}(t-{r\over M})-e_M^{N_0}(t)$, guarantees the stability of the closed-loop system due to ISS of $e_i^{N_0} (1\le i\le M-1)$ systems  with respect to $e_{i+1}^{N_0}$. This is different from the infinite-dimensional closed-loop system \eqref{eq:XMDelayNew}, where the finite-dimensional part of the system is coupled via $\zeta(t)$ with infinite-dimensional tail $z_n (n>N)$. Here the proof of stabilization for any delay $r>0$ provided $M$ and $N$ are large enough becomes challenging.
\end{remark}}

For $L^2$-stability analysis of \eqref{eq:ErrEqSub1} and \eqref{eq:XMDelayNew} we define the Lyapunov functional
\begin{equation}\label{eq:VComponents0New}
	\begin{array}{lll}
		&V(t):=V_0(t)+V_{e}(t)+ V_{Q}(t)+p_e\left|e^{N-N_0}(t)\right|^2,\\
		&V_0(t)=\left|z^{N_0}(t) \right|^2_{P_0}+\sum_{n=N+1}^{\infty}z^2_n(t),\\
		&V_{Q}(t)=Q\int_{t-{r\over M}}^{t} e^{-2\delta(t-s)}\zeta^2(s) ds,\\[0.2cm]
		&V_e(t) = \left|X_e(t) \right|^2_{P_e}+ V_{S_e,X_e}(t)+V_{R_e,X_e}(t)
	\end{array}
\end{equation}
Here $0<P_0\in \mathbb{R}^{(N_0+1)\times (N_0+1)}$, $0<Q,p_e\in \mathbb{R}$, $0<P_e\in \mathbb{R}^{M(N_0+1)\times M(N_0+1)}$ and 
\begin{equation}\label{eq:VComponentsRNew}
	\begin{array}{llll}
		&V_{S_e,X_e}(t):=\int_{t-{r\over M}}^{t} e^{-2\delta(t-s)}\left| X_e(s)\right|_{S_e}^2 ds,\\
		&V_{R_e,X_e}(t):= {r\over M} \int_{-{r\over M}}^{0} \int_{t+\theta}^t e^{-2\delta(t-s)} \left|\dot{X}_e(s)\right|^2_{R_e} ds d\theta,
	\end{array}
\end{equation}
where $0<S_0, R_0 \in \mathbb{R}^{M(N_0+1)\times M(N_0+1)}$. $V_0(t)$ compensates $\zeta(t)$ using \eqref{eq:ZetaEstBoundartActSub}, whereas $V_e(t)$ compensates $\frac{r}{M}$ in the ODEs of the estimation errors. Differentiation of $V_Q(t)$ gives
\begin{equation}\label{eq:VQDiffNew}
	\dot{V}_Q+2\delta V_Q = Q\zeta^2(t)-\epsilon_MQ\zeta^2\left(t-\frac{r}{M}\right), \ \epsilon_M = e^{-2\delta \frac{r}{M}}.
\end{equation}
Differentiating $V_0(t)$ along \eqref{eq:XMDelayNew} we obtain
\begin{equation}\label{eq:VnomDelayNew}
	\begin{array}{lll}
		&\hspace{-2mm}\dot{V}_0+2\delta V_0 =  \left(z^{N_0}\right)^T(t)\left[2\delta P_0+P_0\left(A_0-B_0K_0\right) + \right.\\
		&\hspace{-2mm}\left.\left(A_0-B_0K_0 \right)^TP_0\right]z^{N_0}(t)+2\left(z^{N_0}\right)^T(t)P_0B_0\mathcal{K}_eX_e(t)\\
		&\hspace{-2mm}+2\sum_{n=N+1}^{\infty}\left(-\lambda_n+q+\delta \right)z_n^2(t)\\
		&\hspace{-2mm}+2\sum_{n=N+1}^{\infty}z_n(t)b_n\left[\mathcal{K}_eX_e(t)-K_0z^{N_0}(t)\right].
	\end{array}
\end{equation}	
Let $\alpha_1,\alpha>0$. By the Cauchy-Schwarz inequality we have
\begin{equation}\label{eq:CrosTermNonDelayedNew}
	\begin{array}{lll}
		&2\sum_{n=N+1}^{\infty}z_n(t)b_n\left[\mathcal{K}_eX_e(t)-K_0z^{N_0}(t)\right]\\
		&\leq \left(\frac{1}{\alpha}+\frac{1}{\alpha_1}\right)\sum_{n=N+1}^{\infty}\lambda_nz_n^2(t)+\frac{2\alpha}{N\pi^2}\left|K_0z^{N_0}(t) \right|^2\\
		&\hspace{5mm}+\frac{2\alpha_1}{N\pi^2}\left|\mathcal{K}_eX_e(t) \right|^2
	\end{array}
\end{equation}
where we used the value of $b_n$, given in \eqref{eq:znDelayedSub}, and the estimate
\begin{equation}\label{eq:CrossTermsNew}
	\sum_{n=N+1}^{\infty}\frac{b_n^2}{\lambda_n}\leq \frac{2\alpha}{\pi^2}\int_{N}^{\infty}\frac{dx}{x^2} = \frac{2\alpha}{N\pi^2}.
\end{equation}
Differentiation of  $V_{e}(t)$ and Jensen's inequality lead to
\begin{equation}\label{eq:VeDot}
	\begin{array}{lll}
		&\hspace{-5mm}\dot{V}_{e}+2\delta V_{e} \leq X_e^T(t)\left[P_eF_e+F_e^TP_e+2\delta P_e \right]X_e(t)\\
		&\hspace{-5mm}+2X_e^T(t)P_eG_e\nu_e(t)+2X_e^T(t)P_e\mathcal{L}_e\zeta(t-\frac{r}{M})\\
		&\hspace{-5mm}+2X_e^T(t)P_e\mathcal{L}_eC_1e^{-A_1\frac{r}{M}}e^{N-N_0(t)}+\left| X_e(t)\right|_{S_e}^2-\epsilon_M\times\\
		&\hspace{-5mm}\left[\left|X_e(t)+\nu_e(t) \right|_{S_e}^2+\left|\nu_e(t) \right|_{R_e}^2\right]+\left(\frac{r}{M}\right)^2\left|\dot{X}_e(t) \right|_{R_e}^2.
	\end{array}
\end{equation}
To compensate $\zeta^2(t)$ we use monotonicity of $\lambda_n, \ n\in \mathbb{Z}_+$, \eqref{eq:VnomDelayNew} and \eqref{eq:CrosTermNonDelayedNew} to obtain
\begin{equation}\label{eq:zetaCompensNew}
	\begin{array}{lll}
		&\hspace{-5mm}2\sum_{n=N+1}^{\infty}\left(-\lambda_n+q+\delta +\left[\frac{1}{2\alpha}+\frac{1}{2\alpha_1}\right]\lambda_n\right)z_n^2(t)\\
		&\hspace{-5mm}\overset{\eqref{eq:ZetaEstBoundartActSub}}{\leq} 2\left(-\lambda_{N+1}+q+\delta +\left[\frac{1}{2\alpha}+\frac{1}{2\alpha_1}\right]\lambda_{N+1}\right)\left\|c \right\|_{N}^{-2}\zeta^2(t)
	\end{array}
\end{equation}
provided $-\lambda_{N+1}+q+\delta +\left[\frac{1}{2\alpha}+\frac{1}{2\alpha_1}\right]\lambda_{N+1}\leq 0$. Let
\begin{equation*}\label{eq:etaDelayedSubNew}
	\begin{array}{lll}
		&\eta(t) = \text{col}\{z^{N_0}(t),X_e(t),\nu_e(t),\zeta(t-\frac{r}{M}),e^{N-N_0}(t) \}.
	\end{array}
\end{equation*}
From \eqref{eq:VnomDelayNew}-\eqref{eq:zetaCompensNew} we have
\begin{equation*}\label{eq:FinalIneqNew}
	\begin{array}{ll}
		&\hspace{-5mm}\dot{V}(t)+2\delta V(t)\leq \eta^T(t)\Psi_1 \eta(t) + \frac{2}{\left\|c\right\|_N^2}\Psi_2\zeta^2(t) \leq 0, \ t\geq 0,
	\end{array}
\end{equation*}
provided
\begin{equation*}\label{eq:PointActLMIsDelayedSubNew}
	\begin{array}{lll}
		&\hspace{-3mm}\Psi_2=-\lambda_{N+1}+q+\delta +\left[\frac{1}{2\alpha}+\frac{1}{2\alpha_1}\right]\lambda_{N+1}+\frac{Q\left\|c \right\|_N^{2}}{2} < 0, \\
		&\hspace{-3mm}\Psi_1 = \Psi_{\text{full}}+\left(\frac{r}{M}\right)^2\Lambda^T R_e \Lambda<0.
	\end{array}
\end{equation*}
Here
\begin{equation}\label{eq:PsiMatrix1New}
	\begin{array}{lll}
		&\Psi_{\text{full}} = \left[
		\begin{array}{c|c}
			\Psi_{0} &  \Sigma_1 \\
			\hline
			* &  \Gamma_1
		\end{array}
		\right], \ \Sigma_1 = \scriptsize \begin{bmatrix} 0 \\ 0\\
			P_e\mathcal{L}_eC_1e^{-A_1\frac{r}{M}}\\ 0
		\end{bmatrix}, \\
		&\Gamma_1 =2p_e\left[A_1+\delta I\right], \ \Lambda = [\Lambda_{0}, \mathcal{L}_eC_1e^{-A_1\frac{r}{M}}]
	\end{array}
\end{equation}
with
\begin{equation}\label{eq:PsiMatrix2New}
	\begin{array}{lll}
		&\hspace{-5mm}\Psi_{0} = \scriptsize \left[
		\begin{array}{c|c|c}
			\psi & \begin{matrix}
				P_0B_0\mathcal{K}_e & 0
			\end{matrix} & \begin{matrix}
				0
			\end{matrix} \\
			\hline
			* & \Phi(P_e,S_e,R_e) & P_e\text{col}\left\{\mathcal{L}_e,0\right\}\\
			\hline
			* & * & -\epsilon_MQ
		\end{array}
		\right]\\[0.2cm]
		&\hspace{2mm}+\frac{2}{N\pi^2}\operatorname{diag}\left\{\alpha K_0^TK_0,\alpha_1 \begin{bmatrix}
		\mathcal{K}_e & 0
		\end{bmatrix}^T\begin{bmatrix}
		\mathcal{K}_e & 0
		\end{bmatrix}  ,0\right\},\\
		&\hspace{-5mm} \psi = P_0\left(A_0-B_0K_0 \right)+\left(A_0-B_0K_0 \right)^TP_0+2\delta P_0,\\
		&\hspace{-5mm} \Phi(P_e,S_e,R_e) = \scriptsize\begin{bmatrix}
			P_eF_e+F_e^TP_e+2\delta P_e+(1-\epsilon_M)S & P_eG_e-\epsilon_MS_e\\
			* & -\epsilon_M\left(S_e+R_e \right)\end{bmatrix},\\
		&\hspace{-5mm} \Lambda_0 =[0,F_e,G_e,\mathcal{L}_e].
	\end{array}
\end{equation}
By Schur's complement we have that $\Psi_2<0$ iff
\begin{equation}\label{eq:TailLMI}
	\scriptsize\begin{array}{lll}
		\left[
		\begin{array}{c|c}
			-\lambda_{N+1}+q+\delta+\frac{Q\left\|c \right\|_N^{2}}{2} & 1 \qquad \qquad 1 \\
			\hline
			* &  -2\lambda_{N+1}^{-1}\operatorname{diag}\left\{\alpha,\alpha_1\right\}
		\end{array}
		\right]
	\end{array}\normalsize < 0.\normalsize
\end{equation}
Finally, that  \eqref{eq:N0Sub} yields $\Gamma_1<0$. Therefore, applying Schur complement and taking $p_e\to \infty$ we find that $\Psi_1<0$ iff
\begin{equation}\label{eq:Psireduced}
	\begin{array}{lll}
		&\hspace{-6mm}\Psi_{0} +\left(\frac{r}{M}\right)^2\Lambda_0^T R_e\Lambda_0<0
	\end{array}
\end{equation}
with $\Lambda_0$ given in \eqref{eq:PsiMatrix2New}. Note that \eqref{eq:TailLMI} and \eqref{eq:Psireduced} are reduced-order LMIs whose dimension is independent of $N$. Summarizing, we arrive at:
\begin{theorem}\label{Thm:NonLocalL2Delay}
	Consider \eqref{eq:PDEPointActIntervalDelayedSub}, measurement \eqref{eq:BoundMeasPred} with $c\in L^2(0,1)$ satisfying \eqref{eq:AsscnNonDelayedSub} and control law \eqref{eq:ControllerSub}. Let $\delta>0$ be a desired decay rate. Let $N_0\in \mathbb{Z}_+$ satisfy \eqref{eq:N0Sub} and $N\geq N_0+1$. Assume that $L_0$ and $K_0$ are obtained using \eqref{eq:GainsDesignLSub}  and \eqref{eq:GainsDesignKSub}, respectively. Given $M \in \mathbb{Z}_+, \ M\geq 1$ and $r>0$, let there exist positive definite matrices $P_0,P_e,S_e,R_e$ and scalars $Q, \alpha,\alpha_1>0$ such that \eqref{eq:TailLMI} and \eqref{eq:Psireduced} hold. Then the solution $z(x,t)$ to \eqref{eq:PDEPointActIntervalDelayedSub} under the control law \eqref{eq:ControllerSub} and the corresponding subpredictor-based observer $\hat z(x,t)$ defined by \eqref{eq:obsODENonDelayedSubLastODEs}, \eqref{eq:SubpredODEs} and \eqref{eq:ZhatSeries0Sub}
	satisfy 
	\begin{equation}\label{eq:L2StabilityDelay}
		\begin{aligned}
			&\left\|z(\cdot,t)-\hat{z}(\cdot,t)\right\|+\left\|z(\cdot,t)\right\|\leq De^{-\delta t}\left\|z(\cdot,0)\right\|
		\end{aligned}
	\end{equation}
	for some constant $D>0$.
\end{theorem}

\vspace{0.15cm}
We show next that \eqref{eq:TailLMI} and \eqref{eq:Psireduced} are feasible for any delay $r>0$ provided $M$ and $N$ are large enough. For this purpose consider \eqref{eq:e_M} and
$$V_M(t)=\left|e^{N_0}_M(t) \right|_{P}^2+V_{S,e_M^{N_0}}(t)+V_{R,e_M^{N_0}}(t)$$ with $V_{S,e_M^{N_0}}(t)$, $V_{R,e_M^{N_0}}(t)$ as in \eqref{eq:VComponentsRNew}, where
$0<P,S,R\in \mathbb{R}^{N_0+1}$.
The LMI \begin{equation}\label{eq:LMIeM}
	\begin{array}{lll}
		&\hspace{-4mm}\mathcal{J}(P,S,R)=\scriptsize\begin{bmatrix}
			\phi & -PL_0C_0-e^{-2\delta h}S\\
			* & -e^{-2\delta h}\left(S+R \right)
		\end{bmatrix}\\
		&\hspace{2mm}+h^2 \scriptsize\begin{bmatrix}
			\left(A_0-L_0C_0 \right)^T\\-C_0^TL_0^T\end{bmatrix}\scriptsize R \scriptsize\begin{bmatrix}
			A_0-L_0C_0 & -L_0C_0
		\end{bmatrix}<0,\\
		&\hspace{-4mm}\phi =P(A_0-L_0C_0)+(A_0-L_0C_0)^TP+2\delta P+(1-e^{-2\delta h})S.
	\end{array}
\end{equation}
 with $h={r\over M}$ guarantees $\dot{V}_M(t)+2\delta V_M(t)\leq 0$ along \eqref{eq:e_M}.
  Given $\delta>0$, \eqref{eq:GainsDesignLSub} implies that \eqref{eq:LMIeM} is feasible for small enough $h>0$ (see e.g. \cite{Fridman14_TDS}).

\begin{prop}\label{eq:SubpredGuarantee}
	Given $h>0$, let $0<P,S,R\in \mathbb{R}^{N_0+1}$ such that \eqref{eq:LMIeM} holds.
	Then, given $r>0$ and $M>\frac{r}{h}$, there exists some $N_*$ such that for all $N>N_*$, \eqref{eq:TailLMI} and \eqref{eq:Psireduced} are feasible.
\end{prop}
\begin{proof}
	We first show that there exist $0<P_e,S_e,R_e\in \mathbb{R}^{M(N_0+1)\times M(N_0+1)}$ such that
	\begin{equation}\label{eq:PhiE}
		\Phi(P_e,S_e,R_e)+\left(\frac{r}{M}\right)^2 \begin{bmatrix}
			F_e^T\\ G_e^T
		\end{bmatrix} R_e \begin{bmatrix}F_e & G_e \end{bmatrix} <0
	\end{equation}
	with $\Phi(P_e,S_e,R_e)$ given in \eqref{eq:PsiMatrix2New}. Consider the ODE
	\begin{equation}\label{eq:NominErr}
		\dot{X}_e= F_eX_e(t)+G_e \nu_e(t)
	\end{equation}
	obtained from \eqref{eq:XMDelayNew} by setting $\zeta(t)\equiv 0$ and $e^{N-N_0}(t)\equiv0$. For $V_e(t)$, given in \eqref{eq:VComponents0New}, by standard arguments it can be easily verified that \eqref{eq:PhiE} guarantees $\dot{V}_e(t)+2\delta V_e(t)\leq 0$. We will construct $V_e(t)$ recursively, by using $P$, $S$ and $R$, thereby obtaining $P_e$, $S_e$ and $R_e$. First, consider the ODE  \eqref{eq:e_M}.
	Since $\frac{r}{M}<h$, \eqref{eq:LMIeM} holds with $h$ replaced by $\frac{r}{M}$. Next, consider \eqref{eq:e_M} and the ODE of $e_{M-1}^{N_0}(t)$ in \eqref{eq:NominErr}:
	\begin{equation}\label{eq:e_M-1}
		\begin{array}{lll}
			&\hspace{-4mm}\dot{e}_{M-1}^{N_0}(t) = (A_0-L_0C_0)e_{M-1}^{N_0}(t)-L_0C_0\left(\nu_{e,M-1}(t)-e^{N_0}_M(t)\right)
		\end{array}
	\end{equation}
	Let $\mu>0$ and define
	\begin{equation*}
		V_{M-1}(t) = \left|e^{N_0}_{M-1}(t) \right|_{P}^2+V_{S,e_{M-1}^{N_0}}(t)+V_{R,e_{M-1}^{N_0}}(t)+\mu V_M(t)
	\end{equation*}
	where $V_M(t)$ is \emph{rescaled by $\mu$}. Using \eqref{eq:e_M} and \eqref{eq:e_M-1}, the following LMI guarantees $\dot{V}_{M-1}(t)+2\delta V_{M-1}(t)\leq 0$:
	\begin{equation}\label{eq:LMIM-1}
		\scriptsize \left[
		\begin{array}{c|c}
			\mathcal{J}(P,S,R) & \begin{matrix}
				PL_0C_0+\left(\frac{r}{M}\right)^2(A_0-L_0C_0)^TRL_0C_0\qquad & 0\\-\left(\frac{r}{M}\right)^2C_0^TL_0^TRL_0C_0 \qquad & 0
			\end{matrix}  \\
			\hline
			* & \mu \mathcal{J}(P,S,R)+ \left(\frac{r}{M} \right)^2\begin{bmatrix}C_0^TL_0^TRL_0C_0 & 0 \\ 0 & 0 \end{bmatrix}
		\end{array}
		\right]\normalsize<0.
	\end{equation}
	Since $\mathcal{J}(P,S,R)<0$, taking $\mu$ large enough and applying Schur complement it can be seen that \eqref{eq:LMIM-1} holds. Repeating these arguments by backward induction, i.e choosing
	\begin{equation}
		\begin{array}{lll}
			&V_{M-j}(t) = \left|e^{N_0}_{M-j}(t) \right|_{P}^2+V_{S,e_{M-j}^{N_0}}(t)+V_{R,e_{M-j}^{N_0}}(t)\\
			&\hspace{10mm}+\mu V_{M-j+1}(t), \quad 2\leq j\leq M-1
		\end{array}
	\end{equation}
	and increasing $\mu$ at each step, we obtain that \eqref{eq:PhiE} holds with $W_e = \operatorname{diag}\left\{\mu^{j}W\right\}_{j=0}^{M-1}, \ W\in \left\{P,S,R\right\}$. Next, recall \eqref{eq:TailLMI} and \eqref{eq:Psireduced}, with $\Psi_0$ given in \eqref{eq:PsiMatrix2New}. Set $\alpha = \alpha_1=2$ and let $\beta>0$. Rescaling, we replace $\Phi(P_e,S_e,R_e)$ with $\Phi(\beta P_e,\beta S_e,\beta R_e)=\beta \Phi(P_e,S_e,R_e)$. Let $P_0=P_c$, given in \eqref{eq:GainsDesignKSub}, resulting in $\psi<0$ in \eqref{eq:PsiMatrix2New}. Setting $\beta>0$ to be large enough, then choosing $Q=N$ large enough and applying Schur complement twice in \eqref{eq:Psireduced}, we find that \eqref{eq:TailLMI} and \eqref{eq:Psireduced} hold.
\end{proof}
\subsection{Classical observer-based predictor}
For the case of a classical predictor, we consider a $N+1$ dimensional observer of the form
\begin{equation}\label{eq:ClassPredObs}
	\begin{array}{lll}
		&\hat{z}(x,t)= \hat{z}^{N_0}(t)\cdot \text{col}\left\{\phi_0(x),\dots,\phi_{N_0}(x)\right\}\\
		&\hspace{8mm}+\hat{z}^{N-N_0}(t)\cdot\text{col}\left\{\phi_{N_0+1}(x),\dots,\phi_{N}(x)\right\}.
	\end{array}
\end{equation}
Here $\hat{z}^{N-N_0}(t)$ is defined in \eqref{eq:ObsLastModesSub} and satisfies \eqref{eq:obsODENonDelayedSubLastODEs}, whereas $\hat{z}^{N_0}(t)$ satisfies the following ODE
\begin{equation}\label{eq:HatzN0ClassPred}
	\begin{array}{lll}
		&\dot{\hat{z}}^{N_0}(t)=A_0\hat{z}^{N_0}(t)+B_0u(t-r)\\
		&\hspace{5mm}-L_0\left[C_0\hat{z}^{N_0}\left(t\right)+C_1\hat{z}^{N-N_0}(t)-y(t) \right],\ t\geq 0,\\
		&\hat{z}^{N_0}(t) = 0, \ t\leq 0.
	\end{array}
\end{equation}
Recall $e^{N-N_0}(t)$ given in \eqref{eq:EstErrSub} and satisfying \eqref{eq:ErrEqSub1}. Define $e^{N_0}(t)=z^{N_0}(t)-\hat{z}^{N_0}(t)$, where $z^{N_0}(t)$ is given in
\eqref{eq:ObsLastModesSub}. The innovation term in \eqref{eq:HatzN0ClassPred} can be presented as
\begin{equation*}\label{eq:InnovTermSubClassPred}
	\begin{array}{lll}
		&C_0\hat{z}^{N_0}\left(t\right)+C_1\hat{z}^{N-N_0}(t)-y(t)\\
		&\hspace{10mm} = -C_0e^{N_0}(t)-C_1e^{N-N_0}(t)-\zeta(t)
	\end{array}
\end{equation*}
with $\zeta(t)$, given in \eqref{eq:InnovTermSub}, subject to \eqref{eq:ZetaEstBoundartActSub}. Using these notations with \eqref{eq:znDelayedSub} and \eqref{eq:HatzN0ClassPred}, we obtain
\begin{equation}\label{eq:ODEseN0}
	\begin{array}{lll}
		\dot{e}^{N_0}(t)=\left(A_0-L_0C_0\right)e^{N_0}(t)-L_0C_1e^{N-N_0}(t)-L_0\zeta(t).
	\end{array}
\end{equation}
As in \cite{katzreduced21}, we propose the  predictor-based control law
\begin{equation}\label{eq_Pred}
	\begin{array}{lll}
		&\bar{z}(t)=e^{A_0r}\hat {z}^{N_0}(t)+\int_{t-r}^te^{A_0(t-s)}B_0u(s)ds,\ u(t)=-K_0\bar {z}(t)
	\end{array}
\end{equation}
Note that exponential decay of $\bar{z}(t)$ implies exponential decay of $\hat{z}^{N_0}(t)$ with the same decay rate. Differentiating $\bar {z}(t)$ and using \eqref{eq:HatzN0ClassPred} and \eqref{eq_Pred} we obtain
\begin{equation}\label{ep18e}
	\begin{array}{ll}
		&\hspace{-2mm}\dot{\bar{z}}(t)=\left(A_0-B_0K_0\right)\bar {z}(t)+e^{A_0r}L_0\\
		&\hspace{5mm}\times\Big[C_0e^{N_0}(t)
		+C_1e^{N-N_0}(t)+\zeta(t)\Big], \ t\geq 0 .
	\end{array}
\end{equation}
Then, the reduced-order (decoupled from $\hat{z}^{N-N_0}(t)$) closed-loop system is given by non-delayed ODEs \eqref{eq:ErrEqSub1}, \eqref{eq:ODEseN0}, \eqref{ep18e} and
the tail
\begin{equation}\label{eq:ClassPredTail}
	\begin{array}{lll}
		&\hspace{-5mm}\dot{z}_n(t) =(-\lambda_n+q)z_n(t)-b_nK_0\bar{z}(t-r),\ n>N
	\end{array}
\end{equation}
which depends on the delay. {\color{blue} Note also that in the case of state-feedback (see e.g. \cite{lhachemi2019lmi}), the predictor is given by \eqref{eq_Pred} with $\hat z^{N_0}$ changed by  $z^{N_0}$ leading to decoupled from the tail ODE \eqref{ep18e} with $L_0=0$. The latter simplifies the stability analysis of the closed-loop system and makes the proof of LMI feasibility trivial.}
Next, we consider $L^2$-stability analysis of the closed-loop system, which is delay-independent for $\delta=0$.
Define the Lyapunov functional
\begin{equation}\label{eq:LyapClassPred}
	\begin{array}{lll}
		&\hspace{-2mm}\bar{V}(t) = \bar{V}_0(t)+\left|e^{N_0}(t)\right|_{P_e}^2+\int_{t-r}^te^{-2\delta (t-s)}\left|\bar{z}(s) \right|_S^2 ds,\\
		&\hspace{-2mm}\bar{V}_0(t)= \left|\bar{z}(t)\right|_{P_0}^2+\sum_{n=N+1}^{\infty}z_n^2(t)+p_e\left|e^{N-N_0}(t)\right|^2,
	\end{array}
\end{equation}
where $P_0,P_e,S>0$ are matrices of appropriate dimensions and $p_e>0$ is a scalar. By arguments similar to \eqref{eq:VnomDelayNew}-\eqref{eq:Psireduced} the following LMIs guarantee $\dot{\bar{V}}+2\delta \bar{V}\leq 0$ for $p_e\to \infty$:
\begin{equation}\label{eq:FullLMIClassPres}
	\begin{array}{lll}
&-\epsilon_rS+\frac{2\alpha}{N\pi^2}K_0^TK_0<0,\\[0.1cm]
		&\left[
		\begin{array}{c|c|c}
			\psi+S & P_0e^{A_0r}L_0C_0
			& P_0e^{A_0r}L_0 \\
			\hline
			* & \psi_1 & -P_eL_0\\
			\hline
			* &  * & \psi_2
		\end{array}
		\right]<0,\\[0.3cm]
		&\psi_1 = P_e(A_0-L_0C_0)+(A_0-L_0C_0)^TP_e + 2\delta P_e,\\
		&\psi_2 = \frac{2}{\left\|c\right\|_N^2}\left[-\lambda_{N+1}+q+\delta + \frac{1}{2\alpha}\lambda_{N+1} \right], \ \epsilon_r = e^{-2\delta r},
	\end{array}
\end{equation}
where $\psi$ is given in \eqref{eq:PsiMatrix2New}. Fix any $r>0$. Let $\alpha = 1$, $S = \frac{1}{\sqrt{N}}I$, $P_0$ such that $\psi = -2I$ and $P_e$ such that $\psi_1 = -\beta I$ for $\beta>0$. Choosing first $\beta>0$ large enough and then $N$ large enough and applying Schur complement, it can be verified that \eqref{eq:FullLMIClassPres} holds. Summarizing, we arrive at
\begin{prop}\label{prop:NonLocalL2DelayClassPred}
	Consider \eqref{eq:PDEPointActIntervalDelayedSub}, measurement \eqref{eq:BoundMeasPred} with $c\in L^2(0,1)$ satisfying \eqref{eq:AsscnNonDelayedSub} and control law \eqref{eq_Pred}. Let $\delta>0$ be a desired decay rate. Let $N_0\in \mathbb{Z}_+$ satisfy \eqref{eq:N0Sub} and $N\geq N_0+1$. Let $L_0$ and $K_0$ be obtained using \eqref{eq:GainsDesignLSub}  and \eqref{eq:GainsDesignKSub}, respectively. Given $r>0$, let there exist positive definite matrices $P_0,P_e,S$ and scalar $\alpha>0$ such that LMIs \eqref{eq:FullLMIClassPres} hold. Then the solution $z(x,t)$ to \eqref{eq:PDEPointActIntervalDelayedSub} under the control law \eqref{eq_Pred} and the observer $\hat z(x,t)$ defined by \eqref{eq:ClassPredObs} satisfy \eqref{eq:L2StabilityDelay} for some constant $D>0$. Furthermore, given any $r>0$, the LMI \eqref{eq:FullLMIClassPres} is feasible provided $N$ is large enough.
\end{prop}
\subsection{Numerical example}
We consider \eqref{eq:PDEPointActIntervalDelayedSub} with $q=3$, resulting in an unstable open-loop system. We consider \eqref{eq:BoundMeasPred} with $c(x)=\chi_{[0.3,0.6]}$ (an indicator function). 
We fix $\delta = 0.1$ which results in $N_0=0$.
The controller and observer gains, given by $K_0 = 8.8, \ L_0=14.66$,  are found using \eqref{eq:GainsDesignLSub} and \eqref{eq:GainsDesignKSub}, respectively.

We start with sub-predictors. Given various values of $r>0$, the LMIs of Theorem \ref{Thm:NonLocalL2Delay} were verified for $1\leq N+M\leq 110$ and $1\leq M \leq 20$ by using the standard Matlab LMI toolbox. Table \ref{Tab:Subpred} presents the minimal values of $N$ and $M$
found to guarantee the feasibility (i.e. the exponential stability of the closed-loop system with a decay rate $0.1$). For classical predictors, the LMIs of Proposition \ref{prop:NonLocalL2DelayClassPred} were verified for $1\leq N\leq 100$. Table \ref{Tab:Classpred} presents the minimal values of $N$  which guarantee feasibility of the LMIs. It is seen from the tables that for the same values of $r$, the classical predictor employs a lower-order $N+1$-dimensional observer compared to $(M-1)(N_0+1)+N+1$-dimensional sub-predictors.
Moreover, the simple LMIs for a classical predictor can be verified for any $r$ (with corresponding huge $N$). This is different from sub-predictors, where the  LMIs dimension and the number of  decision variables grow with $M$. This leads to the difficulty of the LMIs verification for large $M$ in Matlab. However, it is well-known that a classical predictor is less friendly for implementation. {\color{blue} The controller uses significant memory as it requires the history of control signals, which turns out to be expensive in implementation. Moreover, one has to carefully discretize the integral term to avoid instabilities. Since a possible instability
	is related to the finite-dimensional part, we suggest to use the same implementation methods as used for ODEs with the classical predictor (see \cite{karafyllis2017predictor} with e.g. sampled-data implementation).}
\vspace{-0.15cm}
\begin{table}[h]
	\begin{center}
		\footnotesize\begin{tabular}{|c|c|c|c|c|c|c|c|}
			\hline
			$r$ & 0.2 &0.4 & 0.6 & 0.8 & 1 & 1.1 & 1.2  \\	
			\hline
			$N$ &4 & 6 & 12 & 22 & 41 & 62 & 90 \\
			\hline
			$M$ & 1 &4 & 7 & 10 & 12 & 13 & 18 \\\hline					
		\end{tabular}
	\end{center}
	\vspace{-0.3cm}
	\caption{\label{Tab:Subpred} Sub-predictors: minimal $N$ and $M$ for  feasibility.}
\end{table}
\vspace{-1cm}
\begin{table}[h]
	\begin{center}
		\footnotesize\begin{tabular}{|c|c|c|c|c|c|c|c|}
			\hline
			$r$ & 0.5 &1 & 1.5 & 2 & 2.3 & 2.5 & 2.8  \\	
			\hline
			$N$ &7 & 12 & 19 & 34 & 42 & 58 & 88 \\
			\hline 					
		\end{tabular}
	\end{center}
	\vspace{-0.3cm}
	\caption{\label{Tab:Classpred} Classical predictor: minimal $N$ for feasibility.}
\end{table}

We carry out simulations of the closed-loop systems. We begin with sub-predictors. Recall the closed-loop system \eqref{eq:obsODENonDelayedSubLastODEs}, \eqref{eq:ControllerSub} and \eqref{eq:XMDelayNew}. We choose the initial condition
\begin{equation}\label{eq:InitCond}
z(x,0)=30x^2(1-x^2), \quad x\in [0,1].
\end{equation}
We consider $r\in \left\{0.6,0.7,0.8\right\}$, $N=12$ and $M=7$. We truncate the tail ODEs after $48$ coefficients and use the approximations $\left\|z(\cdot,t)\right\|_{L^2}^2\approx \sum_{n=0}^{60}z_n^2(t)$,  $\zeta(t)\approx \sum_{n=20}^{60}c_nz_n(t)$. The simulations results are presented in Figure \ref{fig:SimSubspred}.
\begin{figure}
	\centering
	\includegraphics[width=65mm,scale=0.11]{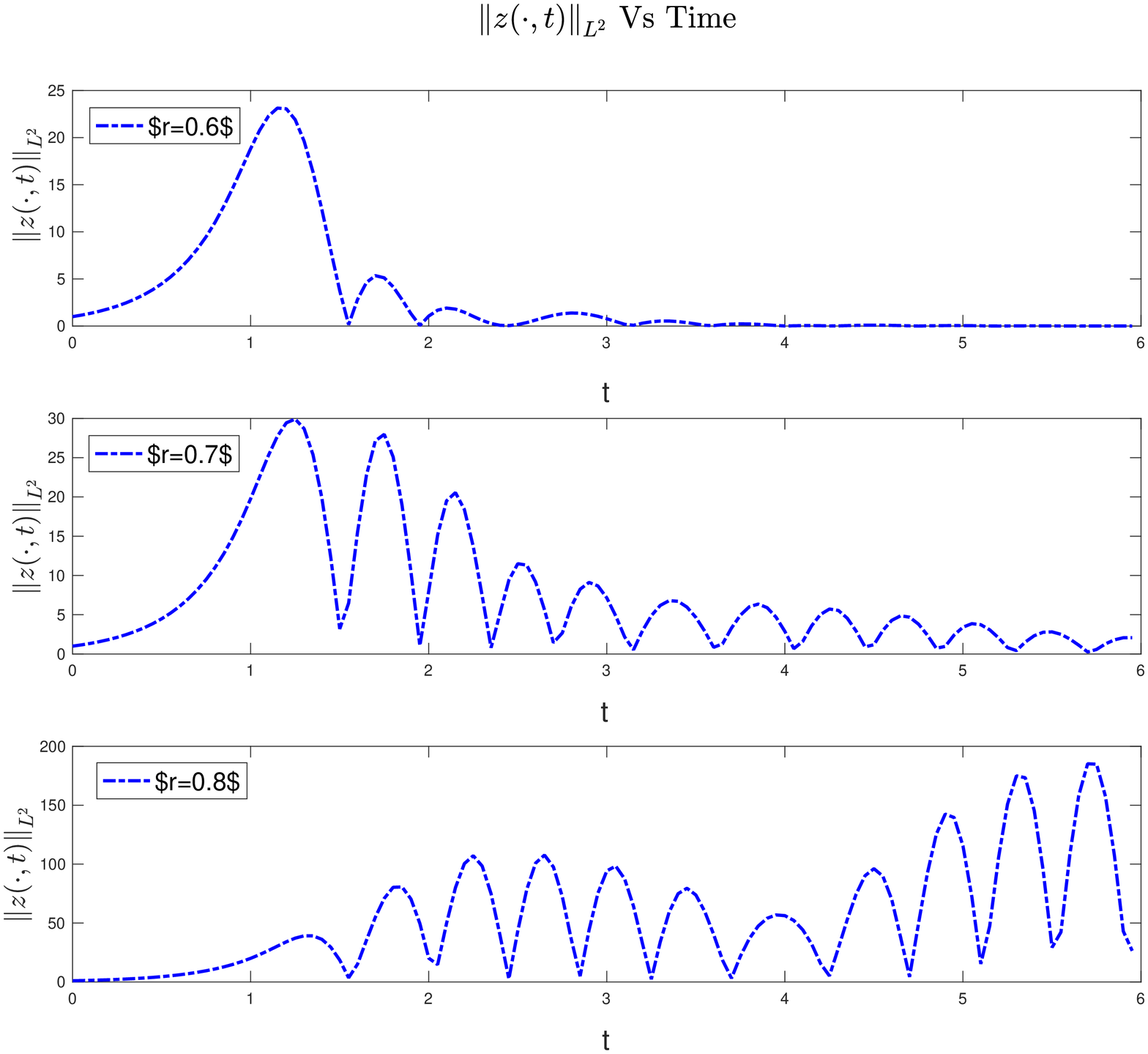}
	\caption{Closed-loop simulations: Sub-Predictors.}\label{fig:SimSubspred}
\end{figure}
For $r=0.6$, $N=12$ and $M=7$, obtained from LMIs (see Table \ref{Tab:Subpred}), the simulations of the closed-loop system confirm the theoretical results. Simulation for larger $r=0.7$ show ultimate boundedness with slow convergence. Simulation with $r=0.8$ shows instability, meaning that our LMIs are only slightly conservative.

Next, we consider classical predictor. Recall the closed-loop system \eqref{eq:obsODENonDelayedSubLastODEs}, \eqref{eq:ErrEqSub1}, \eqref{eq:ODEseN0}, \eqref{eq_Pred} and \eqref{ep18e}. We choose the initial condition \eqref{eq:InitCond}. Let $r\in \left\{1.5,2.5,3\right\}$ and $N=19$. We truncate the tail ODEs after $40$ coefficients and use the approximations $\left\|z(\cdot,t)\right\|_{L^2}^2\approx \sum_{n=0}^{60}z_n^2(t)$,  $\zeta(t)\approx \sum_{n=20}^{60}c_nz_n(t)$. The simulation results are presented in Figure \ref{fig:SimClasspred}. For $r=1.5$ and $N=19$, obtained from LMIs (see Table \ref{Tab:Classpred}), the simulation confirms the theoretical results. For larger $r=2.5$, simulation shows ultimate boundedness, whereas for $r=3$ the closed-loop system is unstable.  For the classical predictor, the lower-order LMIs appeared to be more conservative compared to simulation results. We believe that this is due to the cross-term $e^{A_0r} =e^{3r}$ which couples $\bar{z}(t)$ with $\zeta(t)$ (see \eqref{ep18e}) and becomes huge for larger r. In the LMIs 
\eqref{eq:FullLMIClassPres}, $e^{A_0r}$ appears off-diagonal, leading to difficulty in verifying the feasibility in Matlab.
\begin{figure}
	\centering
	\includegraphics[width=65mm,scale=0.11]{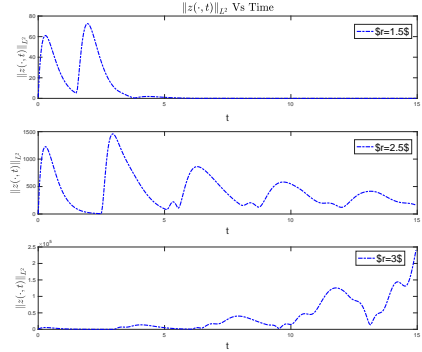}
	\caption{Closed-loop simulations: Classical Predictor.}\label{fig:SimClasspred}
\end{figure}

\section{Conclusion}
We studied constant input delay compensation by finite-dimensional observer-based controllers for the 1D heat equation. We proved that both sub-predictors and classical predictors
theoretically compensate any delay provided the observer dimension  is large. \textcolor{blue}{Classical predictors are known to be less friendly in application to uncertain systems (see e.g. Remark 3 in \cite{zhu2021sub})}. The suggested predictor methods can be
extended in the future to various parabolic PDEs. 

\bibliographystyle{IEEEtran}
\bibliography{IEEEabrv,Bibliography021218}

\end{document}